\documentclass[final,3p,times]{elsarticle}

\usepackage{graphicx}
\usepackage{natbib}
\usepackage{epsfig}
\usepackage{subfigure}
\usepackage{amssymb}
\usepackage{amsthm}
\usepackage{algorithm}
\usepackage{algorithmic}
\usepackage{amsthm}
\usepackage{amsmath,amssymb,amsthm}
\usepackage{latexsym,amsfonts,amscd,amsxtra,amstext}

\newtheorem{theorem}{Theorem}[section]
\newtheorem{lemma}[theorem]{Lemma}

\theoremstyle{definition}
\newtheorem{definition}[theorem]{Definition}
\newtheorem{corollary}[theorem]{Corollary}
\newtheorem{proposition}[theorem]{Proposition}
\theoremstyle{remark}
\newtheorem{remark}[theorem]{Remark}

\journal{Elsevier}
\begin{document}
\begin{frontmatter}



\title{A note on the convergence of nonconvex line search}

\author[a]{Tao Sun\corref{cor}}
\ead{nudtsuntao@163.com}

\author[a,a3]{Lizhi Cheng}
\ead{clzcheng@nudt.edu.cn}

\author[a2]{Hao Jiang}
 \ead{haojiang@nudt.edu.cn}

\cortext[cor]{Corresponding author}

\address[a]{College of Science, National University of Defense Technology,
Changsha, Hunan,  China, 410073.}

\address[a2]{School of Computer, National University of Defense Technology, Changsha, China, 410073.}

\address[a3]{ The State Key Laboratory for High Performance Computation, National University of Defense Technology, Changsha, China, 410073.}




\begin{abstract}
In this note, we consider the line search for a class of abstract nonconvex algorithm which have been deeply studied in the Kurdyka-{\L}ojasiewicz theory. We provide a weak convergence result of the line search in general. When the objective function satisfies the Kurdyka-{\L}ojasiewicz property and some certain assumption, a global convergence result can be derived. An application  is presented for the $\ell_0$ regularized least square minimization in the end of the paper.
\end{abstract}

\end{frontmatter}


\section{Introduction}
Line search is one of  the most frequently used  techniques in the optimization community. In each iteration of this method, a positive scalar $\eta_k$ is computed based on a strategy and the direction $d^k$   is generated by some known way; the point $x^{k+1}$ is then updated by the following scheme
\begin{equation}
    x^{k+1}=x^k+\eta_k d^k.
\end{equation}
However,   the difficulty of   calculating $d^k$ and $\eta_k$ are quite different.   It is usually more difficult to compute $d^k$. Mathematicians have been in search of more efficient strategy for $d^k$, which leads to the development of various quasi-Newton methods \cite{nocedal2006numerical}. While  $\eta_k$ is routinely  obtained by the well-known Armijo search.

In this paper, we consider the line search of   some ``basic algorithm" (\textit{We call it as  Algorithm B through the paper}) for the following optimization problem
\begin{equation}\label{model}
    \min_{x\in \mathbb{R}^N}\Phi(x),
\end{equation}
where $\Phi: \mathbb{R}^N\rightarrow \mathbb{R}$ may be nonconvex and nonsmooth which satisfies that $\inf_{x}\Phi(x)>-\infty$. In the paper, two important  assumptions  about the sequence $\{x^k\}_{k=0,1,2,\ldots}$ generated by Algorithm B are vital to the analysis.

\textbf{A.1}  There exist $\nu>0$ such that $\Phi(x^{k+1})\leq \Phi(x^{k})+\nu\|x^{k+1}-x^k\|_2^2$ for any $k\in \mathbb{N}$.

\textbf{A.2}  There exist $\beta>0$ such that $\textrm{dist}(\textbf{0},\partial \Phi(x^{k+1}))\leq \beta\|x^{k+1}-x^k\|_2$ for any $k\in \mathbb{N}$, where $\partial \Phi(x^{k+1})$  denotes the set of limiting subdifferential of $\Phi$ at $x^{k+1}$ (see Section 2 for a detailed definition).

Actually, various algorithms obey \textbf{A.1} and \textbf{A.2}; specific examples are given in \cite{attouch2013convergence,bolte2014proximal}. It is also proved that Algorithm B converges to a critical point of the objective function under the Kurdyka-{\L}ojasiewicz property assumption\cite{attouch2013convergence}.
In the $k$-th iteration, we consider the line search as

\textbf{Step 1}  Use $x^k$ to generate $y^k$ by Algorithm B.

\textbf{Step 2} Set $d^k=y^k-x^k$. Find $m_k$  as the  smallest   integer number $m$ which obeys that
\begin{equation}
    \Phi(y^{k}+\eta^m d^k)\leq\Phi(y^{k})-\alpha \eta^m\|d^k\|_2^2,
\end{equation}
where $\alpha>0$ and $0<\eta<1$ are parameters (Armijo search). Set $\eta_k=\eta^{m_k}$ if $m_k\leq M$ and $\eta_k=0$ if $m_k>M$, where $M\in \mathbb{N}$ is a parameter. The point $x^{k+1}$ is generated by
\begin{equation}
    x^{k+1}=x^k+(\eta_k+1) d^k.
\end{equation}

The specific scheme of \textbf{Step 1} depends on the scheme of Algorithm B. For example, if Algorithm B is the gradient descend method, $y^{k}=x^k-h\nabla f(x^k)$; if Algorithm B is the forward-backward method, $y^{k}=\textbf{Prox}_g[x^k-h\nabla f(x^k)]$. The Armijo search may not always succeed because $\Phi(y^{k})$ can be the minimum over $\{y^k+t d^k|t\in \mathbb{R}\}$. Therefore, we introduce the parameter $M$.
\begin{algorithm}
\caption{Algorithm B with line search}
\begin{algorithmic}\label{alg1}
\REQUIRE   parameters $\alpha>0$, $0<\eta<1$, $M\in \mathbb{N}$\\
\textbf{Initialization}: $x^0$\\
\textbf{for}~$k=0,1,2,\ldots$ \\
~~~ Use $x^k$ to generate $y^k$ by Algorithm B and set $d^k=y^k-x^k$\\
~~~ Find $m_k$  as the  smallest   integer number $m$ which obeys that $\Phi(y^{k}+\eta^m d^k)\leq\Phi(y^{k})-\alpha \eta^m\|d^k\|_2^2$ \\
~~~ Set $\eta_k=\eta^{m_k}$ if $m_k\leq M$ and $\eta_k=0$ if $m_k>M$\\
~~~ $x^{k+1}=x^k+(\eta_k+1) d^k$\\
\textbf{end for}\\
\end{algorithmic}
\end{algorithm}

In this paper, we consider  a line search strategy for a class of abstract algorithms which obey several certain assumptions.  We prove that any accumulation point of the generated iterations is a critical point of the objective function. If the Kurdyka-{\L}ojasiewicz property and another reasonable assumption are access for the objective function, a global convergence result can be presented.

The rest of this paper is organized as follows. Section 2 introduces the preliminaries. Section 3 proves the convergence. Section 4 considers an application. Section 5   concludes the paper.
\section{Preliminaries}
We collect several definitions as well as some useful properties in variational and convex analysis. Given a lower semicontinuous function $J: \mathbb{R}^N\rightarrow (-\infty,+\infty]$, its domain is defined by
$$dom (J):=\{x\in \mathbb{R}^N: J(x)<+\infty\}.$$
The graph of a real extended valued function $J: \mathbb{R}^N \rightarrow (-\infty, +\infty]$ is defined by
$$\textrm{graph} (J):=\{(x,v)\in\mathbb{R}^N\times \mathbb{R}: v=J(x)\}.$$
The notation of subdifferential plays a central role in (non)convex optimization.

\begin{definition}[subdifferentials\cite{mordukhovich2006variational,rockafellar2009variational}] Let  $J: \mathbb{R}^N \rightarrow (-\infty, +\infty]$ be a proper and lower semicontinuous function.
\begin{enumerate}
  \item For a given $x\in dom (J)$, the Fr$\acute{e}$chet subdifferential of $J$ at $x$, written as $\hat{\partial}J (x)$, is the set of all vectors $u\in \mathbb{R}^N$ which satisfy
  $$\lim_{y\neq x}\inf_{y\rightarrow x}\frac{J(y)-J(x)-\langle u, y-x\rangle}{\|y-x\|_2}\geq 0.$$
When $x\notin dom (J)$, we set $\hat{\partial}J(x)=\emptyset$.

\item The (limiting) subdifferential, or simply the subdifferential, of $J$ at $x\in \mathbb{R}^N$, written as $\partial J(x)$, is defined through the following closure process
$$\partial J(x):=\{u\in\mathbb{R}^N: \exists x^k\rightarrow x, J(x^k)\rightarrow J(x)~\textrm{and}~ u^k\in \hat{\partial}J(x^k)\rightarrow u~\textrm{as}~k\rightarrow \infty\}.$$
\end{enumerate}
\end{definition}
It is easy to verify that the Fr$\acute{e}$chet subdifferential is convex and closed while the subdifferential is closed. When $J$ is convex,  the definition agrees with the one in convex analysis \cite{rockafellar2015convex} as
$$\partial J(x):=\{v: J(y)\geq J(x)+\langle v,y-x\rangle~~\textrm{for}~~\textrm{any}~~y\in \mathbb{R}^N\}.$$
Let $\{(x^k,v^k)\}_{k\in \mathbb{N}}$ be a sequence in $\mathbb{R}^N\times \mathbb{R}$ such that $(x^k,v^k)\in \textrm{graph }(\partial J)$. If $(x^k,v^k)$  converges to $(x, v)$ as $k\rightarrow +\infty$ and $J(x^k)$ converges to $v$ as $k\rightarrow +\infty$, then $(x, v)\in \textrm{graph }(\partial J)$. This indicates the following simple proposition.
\begin{proposition}\label{sublimit}
If $v^k\in \partial J(x^k)$, and $\lim_{k}v^k=v$ and $\lim_{k}x^k=x$. Then, we have that
\begin{equation}
    v\in \partial J(x).
\end{equation}
\end{proposition}
A necessary condition for $x\in\mathbb{R}^N$ to be a minimizer of $J(x)$ is
\begin{equation}\label{Fermat}
0\in \partial J(x).
\end{equation}
When $J$ is convex, (\ref{Fermat}) is also sufficient. A point that satisfies (\ref{Fermat}) is called (limiting) critical point. The set of critical points of $J(x)$ is denoted by $\textrm{crit}(J)$.
\begin{definition}
We call $J$ has a Lipschitz gradient with constant $L_J\geq 0$ if for any $x,y\in dom(J)$
\begin{equation}
\|\nabla J(x)-\nabla J(y)\|_2\leq L_J\|x-y\|_2.
\end{equation}
\end{definition}

\begin{definition}[\cite{bolte2014proximal,attouch2013convergence}]\label{KL}
(a) The function $J: \mathbb{R}^N \rightarrow (-\infty, +\infty]$ is said to have the  Kurdyka-{\L}ojasiewicz property at $\overline{x}\in dom(\partial J)$ if there
 exist $\eta\in (0, +\infty]$, a neighborhood $U$ of $\overline{x}$ and a continuous function $\varphi: [0, \eta)\rightarrow \mathbb{R}^+$ such that
\begin{enumerate}
  \item $\varphi(0)=0$.
  \item $\varphi$ is $C^1$ on $(0, \eta)$.
  \item for all $s\in(0, \eta)$, $\varphi^{'}(s)>0$.
  \item for all $x$ in $U\bigcap\{x|J(\overline{x})<J(x)<J(\overline{x})+\eta\}$, the Kurdyka-{\L}ojasiewicz inequality holds
\begin{equation}
  \varphi^{'}(J(x)-J(\overline{x}))\textrm{dist}(\textbf{0},\partial J(x))\geq 1.
\end{equation}
\end{enumerate}

(b) Proper lower semicontinuous functions which satisfy the Kurdyka-{\L}ojasiewicz inequality at each point of $dom(\partial J)$ are called KL functions.
\end{definition}

\section{Convergence analysis}
\begin{lemma}\label{sc}
Assume that Algorithm B satisfies \textbf{A.1}. Letting $\eta_{+}=\inf_{k}{\eta_k}$\footnote{Obviously, $\eta_{+}\geq 0$.} and $\{x^{k}\}_{k=0,1,2,\ldots}$ be generated by Algorithm 1, it then holds that
\begin{equation}
    \Phi(x^{k+1})\leq\Phi(x^{k})-a\|d^k\|_2^2,
\end{equation}
where $a=\nu+\alpha\eta_+$.
\end{lemma}
\begin{proof}
Note that Algorithm B satisfies \textbf{A.1}, then,
\begin{equation}\label{sctemp1}
     \Phi(y^{k})\leq\Phi(x^{k})-\nu\|d^k\|_2^2.
\end{equation}
And we have that
\begin{equation}\label{sctemp2}
    \Phi(x^{k+1})\leq\Phi(y^{k})-\alpha\eta^k\|d^k\|_2^2.
\end{equation}
Combine (\ref{sctemp1}) and (\ref{sctemp2}), we can derive the result.
\end{proof}

\begin{theorem}\label{th1}
Assume that Algorithm B satisfies \textbf{A.1}. Let $\{x^{k}\}_{k=0,1,2,\ldots}$ be  generated by Algorithm 1, for any accumulation point of $\{x^{k}\}_{k=0,1,2,\ldots}$  is a stationary point of $\Phi$.
\end{theorem}
\begin{proof}
From Lemma \ref{sc}, we can easily see that
\begin{equation}
    \sum_{i=0}^k\|d^k\|_2^2\leq \frac{\Phi(x^{0})-\Phi(x^{k+1})}{a}<+\infty.
\end{equation}
That means
\begin{equation}
    \lim_{k}\|d^k\|_2=0.
\end{equation}
Note that $\|y^{k}-x^k\|_2=\|d^k\|_2$, then,
\begin{equation}\label{th1temp1}
    \lim_{k}\|y^{k}-x^k\|_2=0.
\end{equation}
Assume that $x^*$ is an accumulation point of $\{x^{k}\}_{k=0,1,2,\ldots}$, then there exists $\{k_j\}_{j=0,1,2,\ldots}$ such that
$\lim_{j} x^{k_j}=x^*$.
From (\ref{th1temp1}), we have that $\lim_{j} y^{k_j}=x^*$.
Note \textbf{A.2}, we derive that
\begin{equation}
    \textrm{dist}(\textbf{0},\partial\Phi(y^{k_j}))\leq b\|d^{k_j}\|_2.
\end{equation}
The closedness of $\partial\Phi(y^{k_j})$ means that there exists $v^{k_j}\in \partial\Phi(y^{k_j})$ such that $\|v^{k_j}\|_2\leq b\|d^{k_j}\|_2$.
Note that $\{v^{k_j}\}_{j=0,1,2,\ldots}$ is bounded, without loss of generality, we assume that $\lim_{j} v^{k_j}=v^*$, where $v^*$ is an accumulation point.
From Proposition \ref{sublimit}, we have that $v^*\in \partial \Phi(x^*)$ and $\|v^*\|_2=0$ (i.e., $v^*=\textbf{0}$). Therefore,
\begin{equation}
    \textbf{0}\in \partial \Phi(x^*).
\end{equation}
\end{proof}

\begin{lemma}\label{rl}
Assume that Algorithm B satisfies \textbf{A.2}. If $\Phi$ has a Lipschitz gradient with constant $L_{\Phi}$, for any $k\in \mathbb{N}$, we have that
\begin{equation}
    \textrm{dist}(\textbf{0},\nabla\Phi(x^{k+1}))\leq b\|d^k\|_2,
\end{equation}
where $b=\beta+L_{\Phi}\eta$.
\end{lemma}
\begin{proof}
The closedness of the subdifferential together with  \textbf{A.2} can give that
\begin{eqnarray}
     \textrm{dist}(\textbf{0},\nabla\Phi(x^{k+1}))&\leq& \textrm{dist}(\textbf{0},\nabla\Phi(y^{k}))+\|\nabla\Phi(y^{k})-\nabla\Phi(x^{k+1})\|_2\nonumber\\
     &\leq&\beta\|d^k\|_2+L_{\Phi}\|y^{k}-x^{k+1}\|_2\nonumber\\
     &=&\beta\|d^k\|_2+L_{\Phi}\eta_k\|d^{k}\|_2\nonumber\\
     &\leq&(\beta+L_{\Phi}\eta)\|d^{k}\|_2.
\end{eqnarray}
\end{proof}

\begin{remark}
Actually, in some case, Lemma \ref{rl} holds even function $\Phi$ fails to be differentiable. In Section 4, we present the example.
\end{remark}

\begin{theorem}[Global convergence]\label{th2}
Assume that Algorithm B satisfies \textbf{A.1} and \textbf{A.2}. If   $\Phi$ is a KL function and has a Lipschitz gradient with constant $L_{\Phi}$. Let $\{x^k\}_{k=0,1,2,\ldots}$ generated by Algorithm 1 be bounded, then, $\{x^k\}_{k=0,1,2,\ldots}$ globally converges to a critical point of $\Phi$.
\end{theorem}
\begin{proof}
The scheme of Algorithm 1 certainly gives that
$$\|d^k\|_2\leq\|x^{k+1}-x^k\|_2\leq(1+\eta)\|d^k\|_2.$$
Therefore, Lemmas \ref{sc} and \ref{rl} also indicate that
\begin{equation}\label{th2temp1}
    \widetilde{a}\|x^{k+1}-x^k\|_2^2\leq\Phi(x^{k})-\Phi(x^{k+1}),
\end{equation}
and
\begin{equation}\label{th2temp2}
     \textrm{dist}(\textbf{0},\partial\Phi(x^{k+1}))\leq  \widetilde{b}\|x^{k+1}-x^k\|_2,
\end{equation}
where $\widetilde{a}=a/(1+\eta)$ and $\widetilde{b}=(1+\eta)b$. With (\ref{th2temp1}) and (\ref{th2temp2}),  [Theorem 2.9, \cite{attouch2013convergence}] gives the convergence result of $\{x^k\}_{k=0,1,2,\ldots}$.
\end{proof}

\section{Application}
This part consider the application to the $\ell_0$ regularized least square minimization which reads as
\begin{equation}\label{L0}
   \min_{x\in \mathbb{R}^N} \Phi^0(x):=\frac{1}{2}\|b-Ax\|_2^2+\lambda\|x\|_0.
\end{equation}
The loss function $\frac{1}{2}\|b-Ax\|_2^2$ has a Lipschitz gradient with constant $\|A\|_2^2$. Function $\Phi(x)$ satisfies the Kurdyka-{\L}ojasiewicz property \cite{attouch2013convergence}.     Algorithm B is set as the nonconvex forward-backward algorithm \cite{attouch2013convergence} ( also called as iterative hard thresholding algorithm in \cite{blumensath2008iterative}) which can be presented as
\begin{equation}\label{iht}
x^{k+1}=\textrm{H}_{\frac{\lambda}{h}}(x^{k}-\frac{1}{h} A^{\top}(Ax^k-b)),
\end{equation}
where
\begin{eqnarray}\label{l0proximal}
\textrm{H}_{\frac{\lambda}{h}}(t)=\left\{\begin{array}{ll}
t &\textrm{if} ~|t|\geq\sqrt{\frac{2\lambda}{h}},\\
0&\textrm{otherwise}.
\end{array} \right.
\end{eqnarray}
If $h>\|A\|_2^2$, it has been proved that Algorithm (\ref{iht}) satisfies \textbf{A.1} and \textbf{A.2}. The specific scheme of line search for IHT can be described as follows.
\begin{algorithm}
\caption{IHT with line search}
\begin{algorithmic}\label{alg1}
\REQUIRE   parameters $\alpha,h>\|A\|_2^2$, $0<\eta<1$\\
\textbf{Initialization}: $x^0$\\
\textbf{for}~$k=0,1,2,\ldots$ \\
~~~ $y^{k}=\textrm{H}_{\frac{\lambda}{h}}(x^{k}-\frac{1}{h} A^{\top}(Ax^k-b))$\\
~~~ $d^k=y^k-x^k$\\
~~~ Find $m_k$  as the  smallest   integer number $m$ which obeys that $\Phi^0(y^{k}+\eta^m d^k)\leq\Phi^0(y^{k})-\alpha \eta^m\|d^k\|_2^2$ \\
~~~ Set $\eta_k=\eta^{m_k}$ if $m_k\leq M$ and $\eta_k=0$ if $m_k>M$\\
~~~ $x^{k+1}=x^k+(\eta_k+1) d^k$\\
\textbf{end for}\\
\end{algorithmic}
\end{algorithm}
 Figure 1 reports the function values versus the iteration of IHT and Algorithm 1. We can easily see that line search is quite efficient.
\begin{lemma}
Let $\{x^k\}_{k=0,1,2,\ldots}$  generated by Algorithm 2 be bounded, there exists $b>0$ such that
\begin{equation}
    \textrm{dist}(\textbf{0},\nabla\Phi^0(x^{k+1}))\leq \overline{b}\|d^k\|_2,
\end{equation}
for $k$ sufficiently enough.
\begin{proof}
For simplicity, we denote that $f(x)=\frac{1}{2}\|b-Ax\|_2^2$.
We claim that $\textrm{supp}(y^k)\subseteq \textrm{supp}(x^k)$. Otherwise, there exist $i\in \textrm{supp}(x^k)$ such that $i\notin \textrm{supp}(y^k)$, which means that
\begin{equation}\label{ihtltemp1}
    \|y^k-x^k\|_2\geq |y^k_i-x^k_i|=|y^k_i|\geq \sqrt{\frac{2\lambda}{h}}>0
\end{equation}
from (\ref{l0proximal}). However, (\ref{ihtltemp1}) is a contradiction with (\ref{th1temp1}). Relations (\ref{sctemp1}) and (\ref{sctemp2}) give that $\lim_{k}\Phi^0(y^k)=\lim_{k}\Phi^0(x^k)$. Note that $f$ has  a Lipschitz gradient with constant $\|A\|_2^2$ and $\lim_k\|x^k-y^k\|_2=0$. That means $\lim_k f(x^k)=\lim_k f(y^k)$ and $\lim_k(\|x^k\|_0-\|y^k\|_0)=0$. Note that $\|x^k\|_0-\|y^k\|_0\in \mathbb{N}$, then, $\|x^k\|_0=\|y^k\|_0$ for $k$ sufficiently enough. Therefore, we have that $\textrm{supp}(y^k)=\textrm{supp}(x^k)$ when $k$ is large enough. That is also to say that $\textrm{supp}(y^k)=\textrm{supp}(x^{k+1})$ when $k$ is large enough.

The subdifferential of $\|\cdot\|_0$ can be presented as
\begin{equation}
    [\partial\|x\|_0]_i=\left\{\begin{array}{c}
                                                           \{0\}~~\textrm{if}~~x_i\neq 0, \\
                                                           \mathbb{R} ~~\textrm{if}~~x_i=0.
                                                         \end{array}
    \right.
\end{equation}
Then, we have
$$\textrm{dist}(\textbf{0},\nabla \Phi^0(x^k))=\textrm{dist}(\nabla f(x^k),-\lambda\partial\|x^k\|_0)=\textrm{dist}(\nabla f(x^k),\partial\|x^k\|_0)=\|[\nabla f(x^k)]_{\textrm{supp}(x^k)}\|_2.$$
Note that IHT satisfies \textbf{A.2}, i.e.,
\begin{equation}
  \|[\nabla f(y^k)]_{\textrm{supp}(y^k)}\|_2  \leq b\|d^k\|_2.
\end{equation}
As $k$ is large enough,
\begin{eqnarray}
  \|[\nabla f(x^{k+1})]_{\textrm{supp}(x^{k+1})}\|_2&=&\|[\nabla f(x^{k+1})]_{\textrm{supp}(y^{k})}\|_2\nonumber\\
  &\leq&\|[\nabla f(x^{k+1})-\nabla f(y^{k})]_{\textrm{supp}(y^{k})}\|_2+\|[\nabla f(y^k)]_{\textrm{supp}(y^k)}\|_2\nonumber\\
  &\leq&\|A\|_2^2\|x^{k+1}-y^k\|_2+ b\|d^k\|_2\nonumber\\
 &\leq&(\|A\|_2^2\eta+b)\|x^{k+1}-y^k\|_2.
\end{eqnarray}
\end{proof}
\end{lemma}
\begin{corollary}
Let $\{x^k\}_{k=0,1,2,\ldots}$ generated by Algorithm 2 be bounded, then, $\{x^k\}_{k=0,1,2,\ldots}$ globally converges to a critical point of $\Phi^0(x)$.
\end{corollary}
 \begin{figure}
  \centering
    \includegraphics[width=2.0in]{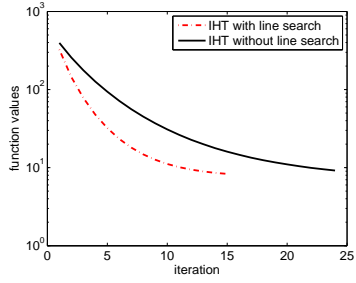}
\caption{Function values via iterations}
\end{figure}
\section{Conclusion}
In this paper, we investigate the convergence of line search for a class of abstract algorithms. The global convergence result is proved under several assumptions. We also consider an application and prove the convergence result in this case.
\section*{Acknowledgments}
We are grateful for the support from the National Natural Science Foundation of Hunan Province, China (13JJ2001), and the Science Project of National University
of Defense Technology (JC120201), and National Science Foundation of China (No.61402495), and National Science Foundation of China (No.61571008).

\small{

\end{document}